\documentclass[review,onefignum,onetabnum]{article}

\usepackage{latexsym}
\usepackage{amssymb,amsbsy,amsmath,amsfonts,amssymb,amscd,amsthm}
\usepackage{subfigure}
\usepackage{graphicx}
\usepackage{hyperref}
\usepackage{dsfont}
\usepackage{mathtools}
\usepackage{pdfsync}
\usepackage{epsfig,epstopdf}
\usepackage{caption}
\usepackage{xcolor}
\usepackage{hyperref}

\newtheorem{prpstn}[]{Proposition}
\newtheorem{rmrk}[]{Remark}

\newcommand{\beq}{\begin{equation}}
\newcommand{\eeq}{\end{equation}}
\newcommand{\be}{\begin{equation}}
\newcommand{\ee}{\end{equation}}
\newcommand{\vep}{\varepsilon}
\DeclareMathOperator*{\argmax}{arg\,max}
\DeclareMathOperator*{\supp}{supp}
\newcommand {\Dt} {\Delta {t}}
\newcommand {\Dx} {\Delta {x}}
\newcommand {\Dy} {\Delta {y}}
\newcommand {\p}   {\partial}
\usepackage{stmaryrd} % enable llb, rrb
\newcommand{\llb}{\llbracket}
\newcommand{\rrb}{\rrbracket}

\newcommand{\bw}{\boldsymbol{w}}
\newcommand{\bn}{\boldsymbol{n}}
\newcommand{\bM}{\boldsymbol{M}}
\newcommand{\brho}{\boldsymbol{\rho}}

\newcommand {\e}  {\varepsilon}

\numberwithin{equation}{section}
\numberwithin{prpstn}{section}
\numberwithin{ass}{section}
\numberwithin{rmrk}{section}

\title{Invasion fronts and adaptive dynamics in a model for the growth of cell populations with heterogeneous mobility}

\author{Tommaso Lorenzi\thanks{Department of Mathematical Sciences ``G. L. Lagrange'', Dipartimento di Eccellenza 2018-2022, Politecnico di Torino, 10129 Torino, Italy (tommaso.lorenzi@polito.it)}
\and
Beno\^it Perthame\thanks{Sorbonne Universit\'e, CNRS, Universit\'e de Paris, Inria, Laboratoire Jacques-Louis Lions UMR7598, F-75005 Paris, France (Benoit.Perthame@sorbonne-universite.fr)}
\and
Xinran Ruan\thanks{Sorbonne Universit\'e, CNRS, Universit\'e de Paris, Inria, Laboratoire Jacques-Louis Lions UMR7598, F-75005 Paris, France  and Department of Mathematical Sciences ``G. L. Lagrange'', Dipartimento di Eccellenza 2018-2022, Politecnico di Torino, 10129 Torino, Italy (xinran.ruan@polito.it)}
}

\begin{document}
\date{}
\maketitle

\begin{abstract}
We consider a model for the dynamics of growing cell populations with heterogeneous mobility and proliferation rate. The cell phenotypic state is described by a continuous structuring variable and the evolution of the local cell population density function ({\it i.e.} the cell phenotypic distribution at each spatial position) is governed by a non-local advection-reaction-diffusion equation. We report on the results of numerical simulations showing that, in the case where the cell mobility is bounded, compactly supported travelling fronts emerge. More mobile phenotypic variants occupy the front edge, whereas more proliferative phenotypic variants are selected at the back of the front. In order to explain such numerical results, we carry out formal asymptotic analysis of the model equation using a Hamilton-Jacobi approach. In summary, we show that the locally dominant phenotypic trait (\emph{i.e.} the maximum point of the local cell population density function along the phenotypic dimension) satisfies a generalised Burgers' equation with source term, we construct travelling-front solutions of such transport equation and characterise the corresponding minimal speed. Moreover, we show that, when the cell mobility is unbounded, front edge acceleration and formation of stretching fronts may occur. We briefly discuss the implications of our results in the context of glioma growth.
\end{abstract}

% REQUIRED
%\begin{keywords}
%continuum mechanics; thin membranes; effective interface conditions; cell invasion; basement membrane; ovarian cancer
%\end{keywords}

% REQUIRED
%\begin{AMS}
%35Q92; 35R05; 92C10; 92C17
%\end{AMS}

\section{Introduction}
\paragraph{Background}
Mathematical models formulated as reaction-diffusion equations with non-local reaction terms have been increasingly used to achieve a more in-depth theoretical understanding of the mechanisms underlying the spatial spread and the phenotypic evolution of populations with heterogeneous motility~\cite{arnold2012existence,berestycki2015existence,bouin2014travelling,bouin2012invasion,bouin2017super,turanova2015model}. 

In these models, the phenotypic state of each individual is described by a continuous structuring variable, and the model itself consists of a balance equation for the local population density function (\emph{i.e.} the phenotypic distribution of the individuals at each spatial position). As is the case for the classical Fisher-KPP model~\cite{fisher1937wave,kolmogorov1937etude}, individuals are assumed to undergo undirected, random movement, which translates into a linear diffusion term. Additionally, intrapopulation variability of individual motility is taken into account by letting the diffusion coefficient be a function of the structuring variable. Moreover, possible changes in individual motility are conceptualised as transitions between phenotypic states, which are modelled through an integral or a differential operator. Finally, in analogy with the non-local version of the Fisher-KPP model~\cite{berestycki2009non,hamel2014nonlocal}, most of these models rely on the assumption that the population undergoes logistic growth at a rate that depends on the local number density of individuals (\emph{i.e.} the integral of the solution with respect to the structuring variable), which is described via a non-local reaction term.

Among these models, the model for the cane toad invasion presented in~\cite{benichou2012front} has received considerable attention from the mathematical community over the last few years. Analysis of this simple yet effective model has made it possible to find a robust mechanistic explanation for the empirical observation that highly motile individuals are, as such, more likely to be found at the edge of the invasion front, and has helped elucidate the way this form of spatial sorting can promote acceleration of the invasion front~\cite{phillips2006invasion,shine2014review,shine2011evolutionary,urban2008toad}. In particular, the existence of travelling-front solutions and the occurrence of spatial sorting in the case of bounded motility has been studied in~\cite{bouin2014travelling,bouin2012invasion,bouin2016bramson,turanova2015model}, while front acceleration in the case of unbounded motility has been investigated in~\cite{berestycki2015existence,bouin2012invasion,bouin2017super}. Furthermore, an evolution equation for the dynamic of the maximum point of the local population density function along the phenotypic dimension (\emph{i.e.} the dominant phenotypic trait) at the edge of the front has been formally derived in~\cite{bouin2012invasion}. 

\paragraph{Content of the paper} We consider a model for the dynamics of growing cell populations with heterogeneous mobility and proliferation rate. In analogy with the models considered in the aforementioned studies, intra-population heterogeneity is here captured by a continuous structuring variable representing the cell phenotypic state and the model consists of a balance equation for the local cell population density function. However, in contrast to the aforementioned studies, such a balance equation takes the form of a non-local advection-reaction-diffusion equation whereby the 
velocity field and the reaction term are both functions of the structuring variable and of the local cell density. This leads to the emergence of invasion fronts with compact support and brings about richer spatio-temporal dynamics of the dominant phenotypic trait throughout the front. 

\paragraph{Outline of the paper}
The remainder of the paper is organised as follows. In Section~\ref{Sec2}, we describe the model and the main underlying assumptions. In Section~\ref{Sec3}, we present the results of numerical simulations, which were obtained using the numerical methods detailed in Appendix~\ref{AppA}. In Section~\ref{Sec4}, we carry out formal asymptotic analysis of the model in order to provide an explanation for such numerical results. In Section~\ref{Sec5}, we discuss the main results of numerical simulations and formal analysis. Moreover, we briefly explain how these mathematical results may shed light on the interplay between spatial sorting and natural selection that underpins tumour growth and the emergence of phenotypic heterogeneity in glioma. Finally, we provide a brief overview of possible research perspectives.   

\section{Statement of the problem}
\label{Sec2}
\paragraph{A model for the dynamics of heterogeneous growing cell populations} We consider a mathematical model for the dynamics of a growing population of cells structured by a variable $y \in [0,Y] \subset \mathbb{R}_+$, which represents the phenotypic state of each cell and takes into account intra-population heterogeneity in cell proliferation rate and cell mobility (\emph{e.g.} the variable $y$ could represent the level of expression of a gene that controls cell proliferation and cell mobility). The population density at position $x \in \mathbb{R}$ and time $t \in [0,\infty)$ is modelled by the function $n(t,x,y)$, the evolution of which is governed by the following non-local partial differential equation (PDE)  
\beq
 \label{eq:PDEn}
 \begin{cases}
\displaystyle{\partial_t n - \alpha \, \mu(y) \, \partial_x \left(n  \, \partial_x \rho(t,x) \right) = R(y,\rho(t,x)) \, n +  \beta \, \partial^2_{yy} n,}
\\\\
\displaystyle{\rho(t,x) := \int_{0}^Y n(t,x,y) \, {\rm d}y,}
\end{cases}
\quad
(x,y) \in \mathbb{R} \times (0,Y),
\eeq
subject to zero Neumann boundary conditions at $y=0$ and $y=Y$.

The second term on the left-hand side of the non-local PDE~\eqref{eq:PDEn} represents the rate of change of the population density due to the tendency of cells to move toward less crowded regions (\emph{i.e.} to move down the gradient of the cell density $\rho(t,x)$)~\cite{Ambrosi_closure,byrne2009individual}. The function $\alpha \, \mu(y)$, with $\alpha > 0$, models the mobility of cells in the phenotypic state $y$. Without loss of generality, we consider the case where higher values of $y$ correlate with higher cell mobility and, therefore, we let $\mu(y)$ be a smooth function that satisfies the following assumptions
\beq
 \label{ass:mu}
\mu(0) > 0, \quad \dfrac{{\rm d} \mu(y)}{{\rm d} y} > 0 \; \text{ for } y \in (0,Y].
\eeq 

Moreover, the first term on the right-hand side of the non-local PDE~\eqref{eq:PDEn} represents the rate of change of the population density due to cell proliferation and death. The function $R(y,\rho(t,x))$ models the fitness (\emph{i.e.} the net proliferation rate) of cells in the phenotypic state $y$ at time $t$ and position $x$ under the local environmental conditions given by the cell density $\rho(t,x)$. We let $R(y,\rho)$ be a smooth and bounded function that satisfies the following assumptions
\beq
 \label{ass:R}
R(Y,0) = 0, \;\; R(0,\rho_M) = 0, \;\; \partial_\rho R(\cdot,\rho) < 0, \;\; \partial_y R(y,\cdot) < 0 \; \text{ for } y \in (0,Y],
\eeq 
with $0<\rho_M < \infty$ being the local carrying capacity of the cell population. Here, the assumption on $\partial_\rho R$ corresponds to saturating growth, while the assumption on $\partial_y R$ models the fact that more mobile cells may be characterised by a lower proliferation rate due to the energetic cost of migration~\cite{aktipis2013life,alfonso2017biology,gallaher2019impact,gerlee2009evolution,gerlee2012impact,giese2003cost,giese1996dichotomy,hatzikirou2012go,orlando2013tumor,pham2012density}. In particular, we will focus on the case where
\beq
\label{def:R}
R(y,\rho) := r(y) - \rho \quad \text{with} \quad r(Y) = 0,  \quad r(0) = \rho_M, \quad \dfrac{{\rm d} r(y)}{{\rm d} y} < 0 \; \text{ for } y \in (0,Y],
\eeq
with $r(y)$ being a smooth and bounded function that models the proliferation rate of cells in the phenotypic state $y$. 

Finally, the second term on the right-hand side of the non-local PDE~\eqref{eq:PDEn} models the effects of spontaneous, heritable phenotypic changes~\cite{huang2013genetic}, which occur at rate $\beta >0$.

\paragraph{Object of study} Focussing on a biological scenario whereby cell movement occurs on a slower time scale compared to cell proliferation and death, while spontaneous, heritable phenotypic changes occur on a slower time scale compared to cell movement~\cite{doerfler2006dna,smith2004measurement,wang2009mathematical}, we introduce a small parameter $\e >0$ and let 
$$
\alpha := \e \quad \text{and} \quad \beta := \e^2. 
$$
Furthermore, in order to explore the long-time behaviour of the cell population (\emph{i.e.} the behaviour of the population over many cell generations), we use the time scaling $t \to t / \e$ in~\eqref{eq:PDEn}, which gives the following non-local PDE for the population density function $n(\frac{t}{\e},x,y) \equiv n_{\e}(t,x,y)$
\beq
 \label{eq:PDEnen}
 \begin{cases}
\displaystyle{\e \, \partial_t n_{\e} - \e \, \mu(y) \, \partial_x \left(n_{\e}  \, \partial_x \rho_{\e}(t,x) \right) = R(y,\rho_{\e}(t,x)) \, n_{\e} +  \e^2 \, \partial^2_{yy} n_{\e},}
\\\\
\displaystyle{\rho_{\e}(t,x) := \int_{0}^Y  n_{\e}(t,x,y) \, {\rm d}y,}
\end{cases}
\quad
(x,y) \in \mathbb{R} \times (0,Y).
\eeq

\newpage
\section{Numerical simulations}
\label{Sec3}
In this section, we report on numerical solutions of the non-local PDE~\eqref{eq:PDEnen} in the case where $R(y,\rho_{\e})$ is defined via~\eqref{def:R}. We choose the following initial condition
\beq
\label{ass:IC}
n_{\e}(0,x,y) = C \, e^{-x^2} \, e^{-\frac{(y-a)^2}{\vep}} \quad \text{with } \; C \; \text{ s.t. } \; C \, \int_0^Y e^{-\frac{(y-a)^2}{\vep}} \, {\rm d}y = 1 \; \text{ and } \; a \in (0,Y),
\eeq
which satisfies $\displaystyle{n_{\e}(0,x,y)  \xrightharpoonup[\e  \rightarrow 0]{\scriptstyle\ast} \rho(0,x) \, \delta_{\bar{y}^0(x)}(y)}$, with $\rho(0,x) = e^{-x^2}$ and $\bar{y}^0(x) \equiv a$. Such an initial condition models a biological scenario whereby $y=a$ is the locally dominant phenotypic trait at every position $x$ at time $t=0$. We use uniform discretisations of steps $\Delta t$, $\Delta x$ and $\Delta y$ of the intervals $(0,T]$, $(0,X)$ and $(0,Y)$, respectively, as computational domains of the independent variables $t$, $x$ and $y$. The implicit finite volume scheme employed to solve numerically~\eqref{eq:PDEnen} complemented with~\eqref{ass:IC} and subject to zero-flux/Neumann boundary conditions at $x=0$ (we expect a constant step), $y=0$ and $y=Y$ is described in Appendix~\ref{sec:scheme}. All numerical computations are performed in {\sc Matlab}.  

\paragraph{Travelling fronts} The plots in Figure~\ref{Fig1} summarise the numerical results obtained in the case where 
\be
\label{sim1}
Y:=1, \quad \mu(y) := y^2 + 0.01, \quad r(y):= 1 - y^2, \quad \rho_M := 1. 
\ee 
The above definitions of $\mu(y)$ and $r(y)$ are such that assumptions~\eqref{ass:mu} and~\eqref{def:R} are satisfied.
\begin{figure}[htbp]
  \includegraphics[width=\textwidth]{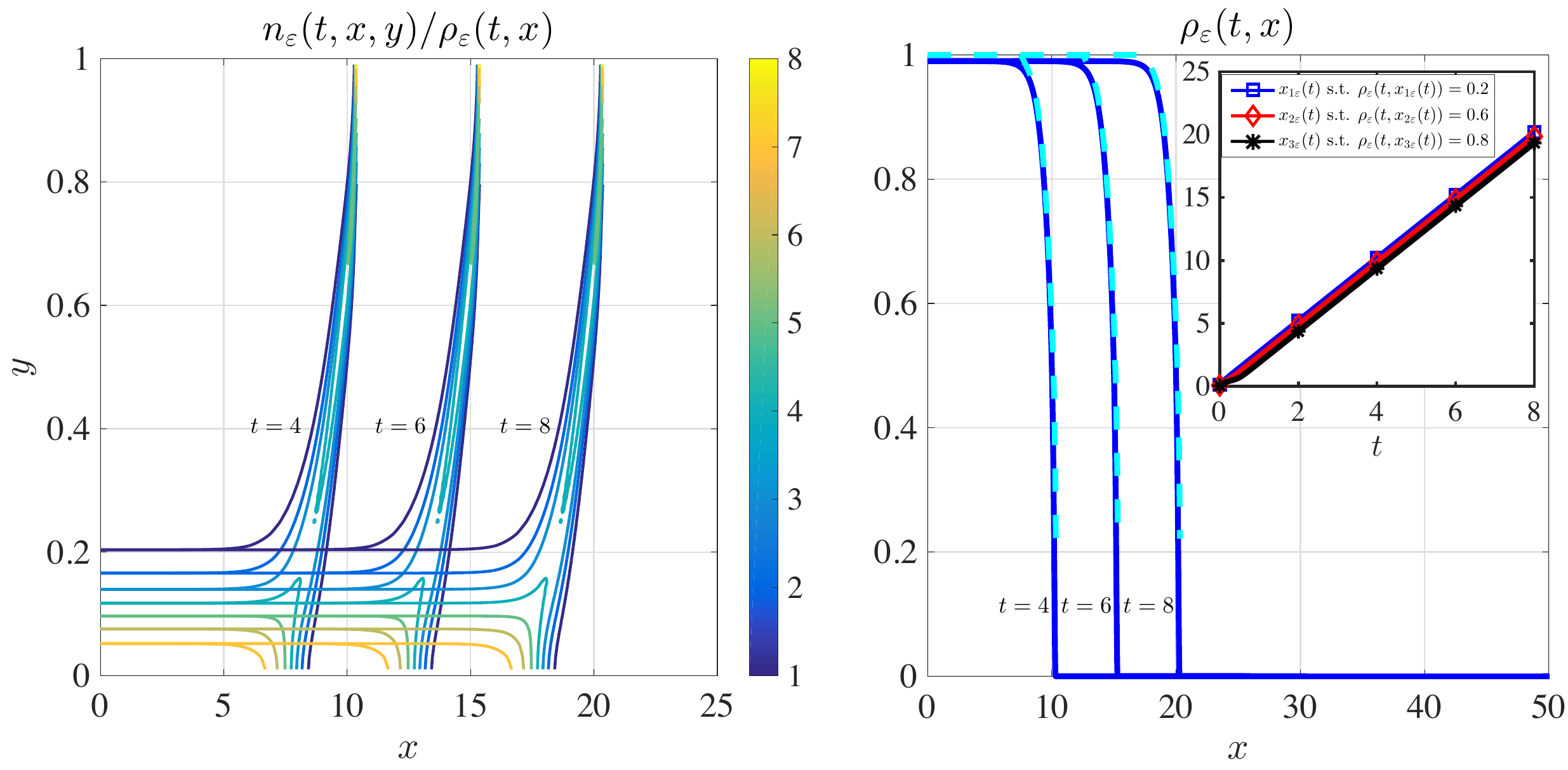}
\caption{{\bf Travelling fronts.} Plots of the normalised cell population density function $n_{\e}(t,x,y)/\rho_{\e}(t,x)$ (left panel) and the cell density $\rho_{\e}(t,x)$ (right panel, solid blue lines) at three successive time instants (\emph{i.e.} $t=4$, $t=6$ and $t=8$). The dashed cyan lines in the right panel highlight the corresponding values of $r(\bar{y}_{{\e}}(t,x))$, with $\bar{y}_{\e}(t,x)$ being the maximum point of $n_{\e}(t,x,y)$ at $x \in \supp(\rho_{\e})$, while the inset of the right panel displays the plots of $x_{1\e}(t)$ (blue squares), $x_{2\e}(t)$ (red diamonds) and $x_{3\e}(t)$ (black stars) such that $\rho_{\e}(t,x_{1\e}(t))=0.2$, $\rho_{\e}(t,x_{2\e}(t))=0.6$ and $\rho_{\e}(t,x_{3\e}(t))=0.8$. These results were obtained solving numerically~\eqref{eq:PDEnen} with $\e := 0.01$ under assumptions~\eqref{def:R},~\eqref{ass:IC} with $a=0.2$, and~\eqref{sim1}. Moreover, $T=8$, $X=25$, $\Dt = 0.01$, $\Dx = 0.01$ and $\Dy = 0.02$.
}
\label{Fig1}
\end{figure}

The left panel of Figure~\ref{Fig1} displays the plots of the normalised cell population density function $n_{\e}(t,x,y)/\rho_{\e}(t,x)$ at three successive time instants (\emph{i.e.} $t=4$, $t=6$ and $t=8$). These plots indicate that for all $x \in \supp(\rho_{\e})$ the normalised population density function $n_{\e}(t,x,y)/\rho_{\e}(t,x)$ is concentrated as a sharp Gaussian with maximum at a point $\bar{y}_{\e}(t,x)$  [\emph{i.e.} $n_{\e}(t,x,y)/\rho_{\e}(t,x) \approx \delta_{\bar{y}_{\e}(t,x)}(y)$ for all $x \in \supp(\rho_{\e})$], and the maximum point $\bar{y}_{\e}(t,x)$ behaves like a compactly supported and monotonically increasing travelling front that connects $y=0$ to $y=Y$. 

The right panel of Figure~\ref{Fig1} displays the plots of the cell density $\rho_{\e}(t,x)$ (solid blue lines) and the function $r(\bar{y}_{{\e}}(t,x))$ (dashed cyan lines) at three successive time instants (\emph{i.e.} $t=4$, $t=6$ and $t=8$). These plots indicate that $\rho_{\e}(t,x)$ behaves like a one-sided compactly supported and monotonically decreasing travelling front that connects $\rho_M$ to $0$. Moreover, there is an excellent quantitative match between $\rho_{\e}(t,x)$ and $r(\bar{y}_{{\e}}(t,x))$, which means that if $\rho_{\e}(t,x)>0$ then the relation $R(\bar{y}_{\e}(t,x),\rho_{\e}(t,x))=0$ holds.

The inset of the right panel of Figure~\ref{Fig1} displays the plots of $x_{1\e}(t)$ (blue squares), $x_{2\e}(t)$ (red diamonds) and $x_{3\e}(t)$ (black stars) such that $\rho_{\e}(t,x_{1\e}(t))=0.2$, $\rho_{\e}(t,x_{2\e}(t))=0.6$ and $\rho_{\e}(t,x_{3\e}(t))=0.8$. These plots show that $x_{1\e}(t)$, $x_{2\e}(t)$ and $x_{3\e}(t)$ are straight lines of slope $\approx 2.5$, which supports the idea that $\rho_{\e}$ behaves like a travelling front of speed $c \approx 2.5$. Such a value of the speed is coherent with the condition on the minimal wave speed $c^*$ given by~\eqref{eq:cminex}. In fact, inserting into~\eqref{eq:cminex} the numerical values of $\bar{y}_{\e}(8,x)$ in place of $\bar{y}(z)$ and the numerical values of $\partial^2_{yy} u_{\e}(8,x,\bar{y}_{\e}(8,x))$ with $u_{\e} = \e \log(n_{\e})$ in place of $\partial^2_{yy} u(z,\bar{y}(z))$ gives $c^* \gtrapprox 2.5$.

\paragraph{Front edge acceleration and stretching fronts} The plots in Figure~\ref{Fig2} summarise the numerical results obtained in the case where 
\be
\label{sim2}
Y:=20, \quad \mu(y) := 0.01 + y^4, \quad r(y):= 1 - \dfrac{y}{1+y}, \quad \rho_M := 1. 
\ee 
The above definitions of $\mu(y)$ and $r(y)$ are chosen so that assumptions~\eqref{ass:mu} and~\eqref{def:R} are satisfied for $Y \to \infty$, and condition~\eqref{eq:cminextoinf} is met (see details below).
\begin{figure}[htbp]
  \includegraphics[width=\textwidth]{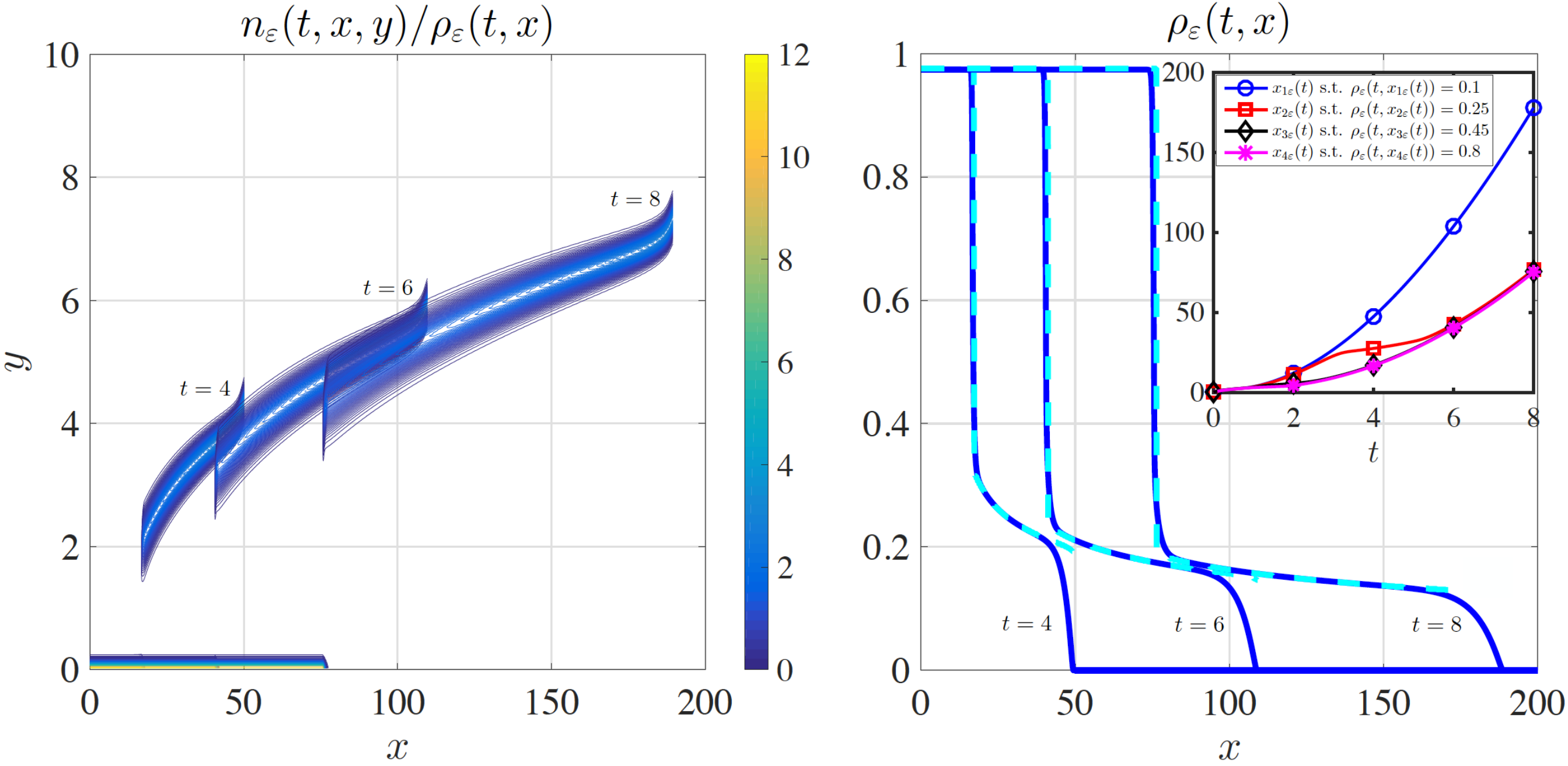}
\caption{{\bf Front edge acceleration and stretching fronts.} Plots of the normalised cell population density function $n_{\e}(t,x,y)/\rho_{\e}(t,x)$ (left panel) and the cell density $\rho_{\e}(t,x)$ (right panel, solid blue lines) at three successive time instants (\emph{i.e.} $t=4$, $t=6$ and $t=8$). The dashed cyan lines in the right panel highlight the corresponding values of $r(\bar{y}_{{\e}}(t,x))$, with $\bar{y}_{\e}(t,x)$ being the maximum point of $n_{\e}(t,x,y)$ at $x \in \supp(\rho_{\e})$, while the inset of the right panel displays the plots of $x_{1\e}(t)$ (blue circles), $x_{2\e}(t)$ (red squares), $x_{3\e}(t)$ (black diamonds) and $x_{4\e}(t)$ (pink stars) such that $\rho_{\e}(t,x_{1\e}(t))=0.1$, $\rho_{\e}(t,x_{2\e}(t))=0.25$, $\rho_{\e}(t,x_{3\e}(t))=0.45$ and $\rho_{\e}(t,x_{4\e}(t))=0.8$. These results were obtained solving numerically~\eqref{eq:PDEnen} with $\e := 0.01$ under assumptions~\eqref{def:R},~\eqref{ass:IC} with $a=0.2$, and~\eqref{sim2}. Moreover, $T=8$, $X=200$, $\Dt = 0.002$, $\Dx = 0.1$ and $\Dy = 0.05$.
}
\label{Fig2}
\end{figure}

The left panel of Figure~\ref{Fig2} displays the plots of the normalised cell population density function $n_{\e}(t,x,y)/\rho_{\e}(t,x)$ at three successive time instants (\emph{i.e.} $t=4$, $t=6$ and $t=8$). Similarly to the case of Figure~\ref{Fig1}, these plots show that for all $x \in \supp(\rho_{\e})$ the normalised population density function $n_{\e}(t,x,y)/\rho_{\e}(t,x)$ is concentrated as a sharp Gaussian with maximum at a point $\bar{y}_{\e}(t,x)$  [\emph{i.e.} $n_{\e}(t,x,y)/\rho_{\e}(t,x) \approx \delta_{\bar{y}_{\e}(t,x)}(y)$ for all $x \in \supp(\rho_{\e})$], and the maximum point $\bar{y}_{\e}(t,x)$ is a monotonically increasing function of $x$ with minimal value $0$ for all $t \in [0,8]$. However, in contrast to the case of Figure~\ref{Fig1}, here $\bar{y}_{\e}(t,x)$ has a jump discontinuity and its maximal value increases as $t$ increases. 

The right panel of Figure~\ref{Fig2} displays the plots of the cell density $\rho_{\e}(t,x)$ (solid blue lines) and the function $r(\bar{y}_{{\e}}(t,x))$ (dashed cyan lines) at three successive time instants (\emph{i.e.} $t=4$, $t=6$ and $t=8$). Similarly to the case of Figure~\ref{Fig1}, these plots indicate that $\rho_{\e}(t,x)$ is a monotonically decreasing function of $x$ with maximal value $\rho_M$ and minimal value $0$ for all $t \in [0,8]$. Furthermore, there is an excellent quantitative match between $\rho_{\e}(t,x)$ and $r(\bar{y}_{{\e}}(t,x))$, which means that if $\rho_{\e}(t,x)>0$ then the relation $R(\bar{y}_{\e}(t,x),\rho_{\e}(t,x))=0$ holds. However, in contrast to the case of Figure~\ref{Fig1}, we have that $\rho_{\e}(t,x)$ behaves like a stretching front, which suggests that the speed of the front edge increases with $t$. 

Coherently with this, the plot of $x_{1\e}(t)$ (blue circles) such that $\rho_{\e}(t,x_{1\e}(t))=0.1$ displayed in the inset of Figure~\ref{Fig2} shows that the value of $x_{1\e}$ undergoes super linear growth, which supports the idea that front edge acceleration occurs. This is also coherent with the fact that, in the case where $\mu(y)$ and $r(y)$ are defined via~\eqref{sim2}, we have that condition~\eqref{eq:cminextoinf} is met and, therefore, the minimal wave speed $c^*$ tends to $\infty$ as $Y \to \infty$.

\section{Formal asymptotic analysis}
\label{Sec4}
In this section, we undertake formal asymptotic analysis of the non-local PDE~\eqref{eq:PDEnen} in order to provide an explanation for the numerical results presented in Section~\ref{Sec3}.

Building on the Hamilton-Jacobi approach presented in~\cite{barles2009concentration,diekmann2005dynamics,lorz2011dirac,perthame2006transport,perthame2008dirac}, we make the real phase WKB ansatz~\cite{barles1989wavefront,evans1989pde,fleming1986pde}
\beq \label{WKB}
n_{\e}(t,x,y) = e^{\frac{u_{\e}(t,x,y)}{\e}},
\eeq
which gives
$$
\partial_t n_{\e} = \frac{\partial_t u_{\e}}{\e} n_{\e}, \quad \partial_x n_{\e} = \frac{\partial_x u_{\e}}{\e} n_{\e}, \quad  \partial^2_{yy} n_{\e} = \left(\frac{1}{\e^2} \left(\partial_y u_{\e} \right)^2 + \frac{1}{\e} \partial^2_{yy} u_{\e} \right) n_{\e}.
$$
Substituting the above expressions into the non-local PDE~\eqref{eq:PDEnen} gives the following Hamilton-Jacobi equation for $u_{\e}(t,x,y)$ 
\beq
\label{eq:PDEue}
\partial_t u_{\e} -  \mu(y) \left(\partial_{x} u_{\e} \, \partial_{x} \rho_{\e} + \e \, \partial^2_{xx} \rho_{\e} \right) = R(y,\rho_{\e}) + \left(\partial_y u_{\e} \right)^2 + \e \, \partial^2_{yy} u_{\e}, \quad (x,y) \in \mathbb{R} \times(0,Y).
\eeq
Letting $\e \to 0$ in~\eqref{eq:PDEue} we formally obtain the following equation for the leading-order term $u(t,x,y)$ of the asymptotic expansion for $u_{\e}(t,x,y)$
\beq
\label{eq:PDEu}
\partial_t u - \mu(y) \,  \partial_{x} \rho \, \partial_{x} u = R(y,\rho) + \left(\partial_y u \right)^2, \quad (x,y) \in \mathbb{R} \times(0,Y),
\eeq
where $\rho(t,x,y)$ is the leading-order term of the asymptotic expansion for $\rho_{\e}(t,x,y)$. 

\paragraph{Constraint on $u$}
Consider $x \in \mathbb{R}$ such that $\rho(t,x) >0$, that is, $x \in \supp(\rho)$, and let $\bar{y}(t,x)$ be a nondegenerate maximum point of $u(t,x,y)$, that is, $\displaystyle{\bar{y}(t,x) \in \argmax_{y \in [0,Y]} u(t,x,y)}$ with $\partial^2_{yy} u(t,x,\bar{y})<0$. Since $R(y,\rho_{\e})$ satisfies assumptions~\eqref{ass:R}, we have that $\rho_{\e}(t,x) < \infty$ for all $\e > 0$. Hence, letting $\e \to 0$ in~\eqref{WKB} formally gives the following constraint for all $t>0$
\beq
\label{eq:ubaryiszero}
u(t,x,\bar{y}(t,x)) = \max_{y \in [0,Y]} u(t,x,y) = 0, \quad x \in \supp(\rho),
\eeq
which also implies that 
\beq
\label{eq:uybaryiszero}
\partial_y u(t,x,\bar{y}(t,x)) = 0 \quad \text{and} \quad \partial_x u(t,x,\bar{y}(t,x)) = 0, \quad x \in \supp(\rho).
\eeq
\begin{rmrk}
The system defined by~\eqref{eq:PDEu} and~\eqref{eq:ubaryiszero} is a constrained Hamilton-Jacobi equation and $\rho(t,x)>0$ can be regarded as a Lagrange multiplier associated with constraint~\eqref{eq:ubaryiszero}. 
\end{rmrk}

\paragraph{Relation between $\bar{y}(t,x)$ and $\rho(t,x)$} Evaluating~\eqref{eq:PDEu} at $y=\bar{y}(t,x)$ and using~\eqref{eq:ubaryiszero} and~\eqref{eq:uybaryiszero} we find 
\beq
\label{eq:Riszero}
R(\bar{y}(t,x),\rho(t,x)) = 0, \quad x \in \supp(\rho).
\eeq
The monotonicity assumptions~\eqref{ass:R} ensure that $\rho \mapsto R(\cdot,\rho)$ and $\bar{y} \mapsto R(\bar{y},\cdot)$ are both invertible. Therefore, relation \eqref{eq:Riszero} gives a one-to-one correspondence between $\bar{y}(t,x)$ and $\rho(t,x)$.

\paragraph{Transport equation for $\bar{y}$}
Differentiating~\eqref{eq:PDEu} with respect to $y$, evaluating the resulting equation at $y=\bar{y}(t,x)$ and using~\eqref{eq:ubaryiszero} and~\eqref{eq:uybaryiszero} yields
\beq
\label{eq:PDEuatbary}
\partial^2_{yt} u(t,x,\bar{y}) - \mu(\bar{y}) \,  \partial_{x} \rho \, \partial^2_{yx} u(t,x,\bar{y})  = \partial_{y} R(\bar{y},\rho), \quad x \in \supp(\rho).
\eeq
Moreover, differentiating~\eqref{eq:uybaryiszero} with respect to $t$ and $x$ we find, respectively,
$$
\partial^2_{ty} u(t,x,\bar{y}) + \partial^2_{yy} u(t,x,\bar{y}) \, \partial_{t} \bar{y}(t,x) = 0 \; \Rightarrow \; \partial^2_{yt} u(t,x,\bar{y}) = - \partial^2_{yy} u(t,x,\bar{y}) \, \partial_{t} \bar{y}(t,x)
$$
and
\beq
\label{eq:u2xy}
\partial^2_{xy} u(t,x,\bar{y}) + \partial^2_{yy} u(t,x,\bar{y}) \, \partial_{x} \bar{y}(t,x) = 0 \; \Rightarrow \; \partial^2_{yx} u(t,x,\bar{y}) = - \partial^2_{yy} u(t,x,\bar{y}) \, \partial_{x} \bar{y}(t,x).
\eeq
Substituting the above expressions of $\partial^2_{yt} u(t,x,\bar{y})$ and $\partial^2_{yx} u(t,x,\bar{y})$ into~\eqref{eq:PDEuatbary} and using the fact that $\partial^2_{yy} u(t,x,\bar{y}) < 0$ gives the following transport equation for $\bar{y}(t,x)$ 
\beq
\label{eq:PDEbary}
\partial_{t} \bar{y} -  \mu(\bar{y}) \, \partial_{x} \rho \, \partial_{x} \bar{y} = \frac{1}{-\partial^2_{yy} u(t,x,\bar{y})} \partial_{y} R(\bar{y},\rho), \quad x \in \supp(\rho),
\eeq
which is a generalised Burgers' equation with source term since $\bar{y}(t,x)$ and $ \rho(t,x)$ are related through \eqref{eq:Riszero}.

\paragraph{Travelling-wave problem} Substituting the travelling-wave ansatz
$$
\rho(t,x) = \rho(z), \quad u(t,x,y) = u(z,y)  \quad \text{and} \quad \bar{y}(t,x) = \bar{y}(z) \quad \text{with} \quad z = x - c \, t, \quad c>0
$$
into~\eqref{eq:PDEu}-\eqref{eq:Riszero} and~\eqref{eq:PDEbary} gives 
\beq
\label{eq:TWu}
- \left(c +  \mu(y) \rho' \right) \partial_z u  = R(y,\rho) + (\partial_y u)^2, \quad (z,y) \in \mathbb{R} \times (0,Y),
\eeq
\beq
\label{eq:ubaryiszeroTW}
u(z,\bar{y}(z)) = \max_{y \in [0,Y]} u(z,y) = 0, \quad \partial_y u(z,\bar{y}(z)) = 0, \quad \partial_z u(z,\bar{y}(z)) = 0, \quad z \in \supp{\left(\rho \right)}, 
\eeq
\beq
\label{eq:TWRiszero}
R(\bar{y}(z),\rho(z)) = 0, \quad z \in \supp{\left(\rho \right)},
\eeq
\beq
\label{eq:TWbary}
- \left(c + \mu(\bar{y}) \rho' \right) \bar{y}'  = \frac{1}{-\partial^2_{yy} u(z,\bar{y})} \partial_{y} R(\bar{y},\rho), \quad z \in \supp{\left(\rho \right)}.
\eeq
We consider travelling-front solutions $\bar{y}(z)$ that satisfy~\eqref{eq:TWbary} subject to the following asymptotic condition
\beq
\label{eq:TWBCy}
\lim_{z \to - \infty}  \bar{y}(z) =0,
%\bar{y}(-\infty) = \argmax_{y \in [0,1]} R(y,\rho(-\infty)).
\eeq
so that, since $R(0,\rho_M)=0$ [{\it cf.} assumptions~\eqref{ass:R}], relation~\eqref{eq:TWRiszero} gives $\displaystyle{\lim_{z \to - \infty} \rho(z) = \rho_M}$.

\paragraph{Monotonicity of travelling-front solutions} 
Differentiating~\eqref{eq:TWRiszero} with respect to $z$ gives 
\beq
\label{eq:TWRziszero}
\partial_y R(\bar{y}(z),\rho(z)) \bar{y}'(z)  + \partial_{\rho} R(\bar{y}(z),\rho(z)) \rho'(z) = 0, \quad z \in \supp{\left(\rho \right)}.
\eeq
Substituting the expression of $\rho'$ given by~\eqref{eq:TWRziszero} into~\eqref{eq:TWbary} yields
$$
-c \ \bar{y}' + \mu(\bar{y}) \ \dfrac{\partial_y R(\bar{y},\rho)}{\partial_{\rho} R(\bar{y},\rho)} \ \left(\bar{y}' \right)^2  = \frac{1}{-\partial^2_{yy} u(z,\bar{y})} \, \partial_{y} R(\bar{y},\rho),
$$
that is,
\beq
\label{eq:TWbary2}
\bar{y}' =  \frac{-\partial_y R(\bar{y},\rho)}{c} \left(\frac{1}{-\partial^2_{yy} u(z,\bar{y})} + \dfrac{\mu(\bar{y}) \left(\bar{y}' \right)^2}{-\partial_{\rho} R(\bar{y},\rho)} \right), \quad z \in \supp{\left(\rho \right)}.
\eeq
Since $\partial^2_{yy} u(z,\bar{y})<0$ and $\partial_y R(y,\cdot)<0$ for $y \in (0,Y]$ [{\it cf.} assumptions~\eqref{ass:R}], using~\eqref{eq:TWbary2} and the expression of $\rho'$ given by~\eqref{eq:TWRziszero} we find
\beq
\label{eq:TWbaryincrhodec}
\bar{y}'(z) > 0 \quad \text{and} \quad \rho'(z) < 0, \quad z \in \supp{\left(\rho \right)}.
\eeq

\paragraph{Position of the front edge} Relation~\eqref{eq:TWRiszero} and monotonicity results~\eqref{eq:TWbaryincrhodec} along with the fact that $R(Y,0)=0$ [{\it cf.} assumptions~\eqref{ass:R}] imply that the position of the edge of a travelling-front solution $\bar{y}(z)$ that satisfies~\eqref{eq:TWbary} subject to asymptotic condition~\eqref{eq:TWBCy} coincides with the unique point $\ell \in \mathbb{R}$ such that $\bar{y}(\ell)=Y$.

\paragraph{Minimal wave speed} Differentiating both sides of~\eqref{eq:TWu} with respect to $y$ gives
$$
- \left(c +  \mu(y) \rho' \right) \partial^2_{yz} u(z,y) -  \dfrac{{\rm d} \mu(y)}{{\rm d} y} \ \rho' \ \partial_{z} u(z,y)  = \partial_y R(y,\rho) + 2 \ \partial_y u(z,y) \ \partial^2_{yy} u(z,y).
$$
Evaluating the above equation at $y=\bar{y}(z)$ using~\eqref{eq:ubaryiszeroTW} yields
\beq
\label{eq:citer}
\left(c + \mu(\bar{y}) \rho'(z) \right) \, \partial^2_{yz} u(z,\bar{y}) + \partial_y R(\bar{y},\rho)  = 0.
\eeq
Moreover, \eqref{eq:u2xy} implies that
$$
\partial^2_{yz} u(z,\bar{y}) = - \partial^2_{yy} u(z,\bar{y}) \, \bar{y}'
$$
and substituting into the latter equation the expression of $\bar{y}'$ given by~\eqref{eq:TWRziszero} we find
$$
\partial^2_{yz} u(z,\bar{y}) = \partial^2_{yy} u(z,\bar{y}) \ \dfrac{\partial_{\rho} R(\bar{y},\rho)}{\partial_y R(\bar{y},\rho)} \ \rho'(z).
$$
Inserting the above expression of $\partial^2_{yz} u(z,\bar{y})$ into~\eqref{eq:citer} gives 
$$
\mu(\bar{y}) \  \partial^2_{yy} u(z,\bar{y}) \ \partial_{\rho} R(\bar{y},\rho) \left(\rho'\right)^2 + c \, \partial^2_{yy} u(z,\bar{y}) \ \partial_{\rho} R(\bar{y},\rho) \ \rho' + \left(\partial_y R(\bar{y},\rho)\right)^2 = 0.
$$
In the case where $R(y,\rho)$ is defined via~\eqref{def:R}, we have that 
$$
\partial_{\rho} R(\cdot,\rho) = -1 \quad \text{and} \quad \displaystyle{\partial_y R(\bar{y},\cdot) = \dfrac{{\rm d} r(\bar{y})}{{\rm d} y}}.
$$
Hence, the latter equation becomes
\beq
\label{eq:precex}
\mu(\bar{y}) \  \partial^2_{yy} u(z,\bar{y}) \ \left(\rho'\right)^2 + c \, \partial^2_{yy} u(z,\bar{y}) \ \rho' - \left(\dfrac{{\rm d} r(\bar{y})}{{\rm d} y}\right)^2 = 0.
\eeq
Coherently with~\eqref{eq:TWbaryincrhodec}, the real roots of~\eqref{eq:precex} seen as a quadratic equation for $\rho'$ are negative. Furthermore, the following condition has to hold for the roots to be real 
$$
2 \ \left|\dfrac{{\rm d} r(\bar{y})}{{\rm d} y}\right| \sqrt{\dfrac{\mu(\bar{y})}{\left|\partial^2_{yy} u(z,\bar{y})\right|}} \leq c.
$$
This indicates that there is a minimal wave speed $c^*$, which satisfies the following condition
\beq
\label{eq:cminex}
c^* \geq \sup_{z \in \supp(r(\bar{y}))} 2 \ \left|\dfrac{{\rm d} r(\bar{y}(z))}{{\rm d} y}\right| \sqrt{\dfrac{\mu(\bar{y}(z))}{\left|\partial^2_{yy} u(z,\bar{y}(z))\right|}},
\eeq
where we have used the fact that, when $R(y,\rho)$ is defined via~\eqref{def:R}, relation~\eqref{eq:TWRiszero} gives
$$
%\label{eq:TWrhoex}
\rho(z) \equiv r(\bar{y}(z)), \quad z \in \supp{\left(\rho \right)}.
$$
Condition~\eqref{eq:cminex} implies that if
\beq
\label{eq:cminextoinf}
\left|\dfrac{{\rm d} r(Y)}{{\rm d} y}\right|  \sqrt{\mu(Y)} \longrightarrow \infty \; \text{ as } \; Y \to \infty
\eeq
then $c^* \to \infty$ as $Y \to \infty$.

\section{Discussion, biological implications and research perspectives}
\label{Sec5}

\paragraph{Discussion of the main results}
In this paper, we have reported on the results of numerical simulations of the non-local PDE~\eqref{eq:PDEnen} complemented with~\eqref{ass:mu} and~\eqref{def:R}, and subject to zero Neumann boundary conditions at $y=0$ and  $y=Y$. These numerical results indicate that 
\beq
\label{eq:convintro}
\text{if } n_{\e}(0,x,y)  \xrightharpoonup[\e  \rightarrow 0]{\scriptstyle\ast} \rho(0,x) \, \delta_{\bar{y}^0(x)}(y) \; \text{ then } \; n_{\e}(t,x,y)  \xrightharpoonup[\e  \rightarrow 0]{\scriptstyle\ast} \rho(t,x) \, \delta_{\bar{y}(t,x)}(y),
\eeq
with $\rho(t,x)$ and $\bar{y}(t,x)$ such that if $\rho(t,x)>0$ then the relation $R(\bar{y}(t,x), \rho(t,x)) =0$ holds. These numerical results also indicate that in the case where $Y \in \mathbb{R}^*_+$ (\emph{i.e.} when $\mu(Y)<\infty$ and, therefore, the cell mobility is bounded), $\rho(t,x)$ in~\eqref{eq:convintro} behaves like a one-sided compactly supported and monotonically decreasing travelling front $\rho(z) \equiv \rho(x-ct)$ that connects $\rho_M$ to $0$, while $\bar{y}(t,x)$ in~\eqref{eq:convintro} behaves like a compactly supported and monotonically increasing travelling front $\bar{y}(z) \equiv \bar{y}(x-ct)$ that connects $0$ to $Y$. Furthermore, we have provided numerical evidence for the fact that front edge acceleration and formation of stretching fronts may occur in the case where $Y \to \infty$ (\emph{i.e.} when $\mu(Y) \to \infty$ and, therefore, the cell mobility is unbounded).

In order to explain such numerical results, we have undertaken formal asymptotic analysis of the non-local PDE~\eqref{eq:PDEnen} complemented with~\eqref{ass:mu} and~\eqref{ass:R} in the asymptotic regime $\e \to 0$ using a Hamilton-Jacobi approach. In particular, we have shown that $\bar{y}(t,x)$ satisfies a generalised Burgers' equation with source term [see transport equation~\eqref{eq:PDEbary}] and $\rho(t,x)=R(\bar{y}(t,x),\rho(t,x))^{-1}(0)$ [see relation~\eqref{eq:Riszero}]. Moreover, we have shown that travelling-front solutions $\bar{y}(z)$ of such transport equation which connect $0$ to $Y$ are monotonically increasing, whilst the corresponding $\rho(z) = R(\bar{y}(z),\rho(z))^{-1}(0)$ is monotonically decreasing and connect $\rho_M$ to $0$ [see the monotonicity results given by~\eqref{eq:TWbaryincrhodec}]. Finally, in the case where $R(y,\rho)$ is defined via~\eqref{def:R}, we have characterised the minimal speed $c^*$ of such travelling-front solutions [see the result given by~\eqref{eq:cminex}] and derived sufficient conditions under which $c^* \to \infty$ as $Y \to \infty$ [see condition~\eqref{eq:cminextoinf}].

\paragraph{Biological implications of the main results} From a biological point of view, $\bar{y}(t,x)$ represents the dominant phenotypic trait at position $x$ and time $t$ and the transport equation for $\bar{y}(t,x)$ can be seen as a generalised canonical equation of adaptive dynamics~\cite{dieckmann1996dynamical,diekmann2005dynamics}, which describes the spatio-temporal evolution of the dominant phenotypic trait. Furthermore, the fact that $\rho(t,x)$ behaves like a monotonically decreasing travelling front $\rho(z)$ that connects $\rho_M$ to $0$ represents the formation of an invasion front of cells that expands into the surrounding environment~\cite{perez2011bright}. Hence, the fact that $\bar{y}(t,x)$ behaves like a monotonically increasing travelling front $\bar{y}(z)$ that connects $0$ to $Y$ has the following biological implications. First, the fact that the front $\bar{y}(z)$ is monotonic indicates that cells with different phenotypic characteristics populate different parts of the invasion front -- \emph{i.e.} phenotypic heterogeneity is dynamically maintained throughout the front. Secondly, since larger values of $y$ correlate with a lower proliferation rate and a higher mobility, the fact that the front $\bar{y}(z)$ is increasing indicates that more mobile/less proliferative phenotypic variants occupy the front edge, whereas less mobile/more proliferative phenotypic variants are selected at the back of the front. This recapitulates previous theoretical and experimental results on glioma growth, which indicate that the interior of the tumour consists mainly of proliferative cells while the tumour border comprises mainly cells that are more mobile and less proliferative -- see, for instance, ~\cite{alfonso2017biology,dhruv2013reciprocal,giese2003cost,giese1996dichotomy,wang2012ephb2,xie2014targeting} and references therein.

\paragraph{Research perspectives}
Building upon the results presented in this paper, a number of generalisations of the mathematical model given by the non-local PDE~\eqref{eq:PDEn} could be considered in order to investigate the role of the concerted action between evolutionary and mechanical processes in tissue development and tumour growth. For example, a natural generalisation is the one given by the following non-local PDE
\beq
 \label{eq:PDEnn}
 \begin{cases}
\displaystyle{\partial_t n - \mu(y) \, \nabla_{{\bf x}} \cdot \left(n  \, \nabla_{{\bf x}} P(t,{\bf x}) \right) = R(y,P(t,{\bf x})) \, n +  \beta \, \partial^2_{yy} n,}
\\\\
\displaystyle{P \equiv \Pi(\rho), \quad \rho(t,{\bf x}) := \int_{0}^Y n(t,{\bf x},y) \, {\rm d}y,}
\end{cases}
\;
({\bf x},y) \in \mathbb{R}^{d} \times (0,Y)
\eeq
subject to zero Neumann boundary conditions at $y=0$ and $y=Y$. Here, $d=1,2,3$ depending on the biological problem considered, and the function $P(t,{\bf x})$ is the pressure exerted by cells at position ${\bf x}$ and time $t$, which is defined via the barotropic relation $\Pi(\rho)$ that satisfies suitable assumptions. 

On the basis of the knowledge we have here acquired on the behaviour of the solutions to the non-local PDE~\eqref{eq:PDEn}, under asymptotic scenarios relevant to applications we may expect $n(t,{\bf x}, y)$ to converge to a singular measure of the form $\rho(t,{\bf x}) \delta_{\bar{y}(t,{\bf x})}(y)$. Moreover, depending on the choices of $Y$, $\mu(y)$, $R(y,P(t,{\bf x}))$ and $\Pi(\rho)$, the cell density $\rho(t,{\bf x})$ may develop into an invading front or it may exhibit interface instabilities~\cite{kim2020interface,lorenzi2016interfaces,pham2018nonlinear,tang2014composite}. Finally, when the following definition of $\Pi$ is considered
$$
\Pi(\rho) := K_{\gamma} \, \rho^{\gamma}, \qquad  K_{\gamma} >0, \; \gamma > 1,
$$
which was proposed in~\cite{perthame2014hele} in order to capture key aspects of tumour and tissue growth while ensuring analytical tractability of the model equation, one finds that $P(t,{\bf x})$ satisfies a porous medium-type equation. Hence, free-boundary problems may emerge in the asymptotic regime $\gamma \to \infty$ (\emph{i.e.} the asymptotic regime whereby cells are regarded as an incompressible fluid). These are lines of research that we will be pursuing in the near future.

\newpage
\appendix

\section{Numerical methods} \label{sec:scheme}
\label{AppA}
%The implicit finite volume scheme
Since $\rho_{\e}(t,x)$ might develop into a stiff travelling front, solving the non-local PDE~\eqref{eq:PDEnen} via an explicit finite volume scheme would result in a severe CFL constraint on $t$. In order to overcome such a limitation, we carried out numerical simulations using the implicit finite volume scheme presented here. For simplicity of notation, throughout this appendix we drop the subscript $\e$.

\paragraph{Time splitting} Adopting a time-splitting approach, which is based on the idea of decomposing the original problem into simpler subproblems that are then sequentially solved at each time-step, we decompose the non-local PDE~\eqref{eq:PDEnen} posed on $\Omega := (0,T]\times(0, X)\times(0,Y)$ into two parts -- {\it viz.} the diffusion-advection part corresponding to the following non-local PDE
\be\label{dyn:convection}
\begin{cases}
\partial_t n -\p_x(\mu(y) \, n \, \p_x\rho)=\vep \, \p_{yy}^2 n,
\\\\
\displaystyle{\rho(t,x) := \int_0^Y n(t,x,y) \, {\rm d}y}
\end{cases}
\ee
and the reaction part corresponding to the following integro-differential equation
\be\label{dyn:grow}
\begin{cases}
\vep \, \partial_t n = R(y,\rho) \, n,
\\\\
\displaystyle{\rho(t,x) := \int_0^Y n(t,x,y) \, {\rm d}y}.
\end{cases}
\ee
We complement~\eqref{dyn:convection} with zero-flux/Neumann boundary conditions at $x=0$ (we expect a constant step), $y=0$ and $y=Y$. Note that making the ansatz $n(t,x,y) = e^{\frac{u(t,x,y)}{\vep}}$, as similarly done in Section~\ref{Sec4}, the integro-differential equation~\eqref{dyn:grow} can be rewritten in the following alternative form
\be\label{dyn:grow_u}
\begin{cases}
\p_t u = R(y, \rho),
\\\\
\displaystyle{\rho(t,x) := \int_0^Y e^{\frac{u(t,x,y)}{\vep}} \, {\rm d}y}.
\end{cases}
\ee

\paragraph{Preliminaries and notation} We denote by $\llb k_1, k_2\rrb$ the set of integers between $k_1$ and $k_2$. We discretise $\Omega$ via a uniform structured grid of steps $\Dt$, $\Dx$, $\Dy$ whereby $t_h=h\Dt$ and the $(j, k)$-th cell is
\be\label{notation:x}
K_{j-\frac12,k-\frac12} = (x_{j-1},x_{j})\times (y_{k-1},y_{k})\quad  \text{with}\quad  x_j =  j\Dx, \quad y_k = k \Dy, 
\ee
where $j \in \llb1,  m_x\rrb$ and $k\in \llb1,  m_y\rrb$, 
$\Dx = \frac{X}{m_x}$, $\Dy = \frac{Y}{m_y}$ and $m_x,m_y\in \mathbb{N}$. Moreover, we let $N_{j-\frac12,k-\frac12}^h$ be the numerical approximation of the average of $n(t_h, x, y)$ over the cell $K_{j-\frac12,k-\frac12}$ and we consider the following first-order approximation of the average of $\rho(t_h, x)$ over the interval $(x_{j-1},x_{j})$
$$
\rho_{j-\frac12}^{h} = \Dy \sum_{k=1}^{m_y} N_{j-\frac12,k-\frac12}^{h}.
$$
Finally, we introduce the notation
\be
%\tilde{N}_{k(1+m_x)+j}^h
%$\tilde{N}_{k(1+m_x)+j}^h = N_{j-\frac12,k-\frac12}^h$,  
\bn^h = \left(N_{j-\frac12,k-\frac12}^h\right)^{\rm T}\in \mathbb{R}^{(m_x+1), (m_y+1)}, \quad \brho^{h} = \left(\rho_{j-\frac12}^h\right)^{\rm T} \in \mathbb{R}^{m_x}
\ee 
with $j \in\llb 1, m_x\rrb$ and $k \in\llb 1, m_y\rrb$.
% and $l \in\llb1, m_x\rrb$. 

\paragraph{Numerical scheme} {\bf Step 1} We first solve numerically~\eqref{dyn:convection} by using the following implicit scheme
\be\label{scheme:conv}
\frac{N_{j-\frac12,k-\frac12}^{*} - N_{j-\frac12,k-\frac12}^h}{\Dt} - \frac{1}{\Dx} \left[ 
F_{j,k-\frac12}^{*}
- 
F_{j-1,k-\frac12}^{*}
\right] 
=
\vep \frac{N_{j-\frac12,k+\frac12}^{*} - 2 N_{j-\frac12,k-\frac12}^{*} + N_{j-\frac12,k-\frac32}^{*}}{\Dy^2},
\ee
where $F_{j, k-\frac12}^{*}$ represents the numerical flux at the boundary $\p K_{j-\frac12,k-\frac12} \cap \{x=x_{j}\}$, which is given by the following upwind approximation 
\be
F_{j, k-\frac12}^{*} =  \mu_{k-\frac12}\left[ 
-(\delta_x\rho_{j}^{*})_{-} \, N_{j-\frac12,k-\frac12}^{*} 
+
(\delta_x\rho_{j}^{*})_{+} \, N_{j+\frac12,k-\frac12}^{*}
\right].
\ee 
Here, $\mu_{k-\frac12} = \mu(y_{k-\frac12})$,
$$
\delta_x\rho_j^{*}=\dfrac{\left(\rho_{j+\frac12}^{*} - \rho_{j-\frac12}^{*}\right)}{\Dx} \quad \text{with} \quad \rho_{j-\frac12}^{*} = \Dy \sum_{k=1}^{m_y} N_{j-\frac12,k-\frac12}^{*},
$$
and $(\cdot)_-$ and $(\cdot)_+$ are, respectively, the negative and positive part of $(\cdot)$. Analogous considerations hold for $F_{j-1,k-\frac12}^{*}$. We complement~\eqref{scheme:conv} with boundary conditions corresponding to zero-flux/Neumann boundary conditions at $x=0$ (we expect a constant step), $y=0$ and $y=Y$.
\\\\
{\bf Step 2} We solve numerically~\eqref{dyn:grow_u} using the following implicit scheme 
\be\label{scheme:rho_growth}
\frac{U_{j-\frac12,k-\frac12}^{h+1} - U_{j-\frac12,k-\frac12}^*}{\Dt} = R\left(y_{k-\frac12}, \rho_{j-\frac12}^{h+1}\right),
\ee
where $U_{j-\frac12,k-\frac12}^* = \e \ln \left(N_{j-\frac12,k-\frac12}^{*} \right)$ and $N_{j-\frac12,k-\frac12}^{*}$ is obtained via~\eqref{scheme:conv}. Since
\begin{eqnarray}\label{eq:rho_growth}
\rho_{j-\frac12}^{h+1} &=& \Dy \sum_{k=1}^{m_y} \exp\left( \dfrac{U_{j-\frac12,k-\frac12}^{h+1}}{\vep}\right) \nonumber\\ 
&=& \Dy \sum_{k=1}^{m_y} \exp\left( \dfrac{U_{j-\frac12,k-\frac12}^{*}+\Dt R\left(y_{k-\frac12}, \rho_{j-\frac12}^{h+1}\right)}{\vep}\right),
\end{eqnarray}
in the case where the function $R$ is defined via~\eqref{def:R} the value of $\rho_{j-\frac12}^{h+1}$ can be found by solving~\eqref{eq:rho_growth}.   The value of $\rho_{j-\frac12}^{h+1}$ so obtained is substituted into~\eqref{scheme:rho_growth}, which is then solved in order to find $U_{j-\frac12,k-\frac12}^{h+1}$, whose value is finally used to compute $N_{j-\frac12,k-\frac12}^{h+1}$ via the formula
$$
N_{j-\frac12,k-\frac12}^{h+1} =  \exp\left(\frac{U_{j-\frac12,k-\frac12}^{h+1}}{\vep} \right).
$$

\paragraph{Properties of the numerical scheme~\eqref{scheme:conv}}
Due to the the strong coupling between $n(t,x,y) $ and $\rho(t,x)$ in the non-local PDE~\eqref{dyn:convection}, it remains an open problem to prove existence and uniqueness of the solution to the corresponding initial-boundary value problem. Similarly, proving unique solvability of the nonlinear, nonlocal, implicit scheme~\eqref{scheme:conv} remains an open problem. 

Here, assuming solvability of~\eqref{scheme:conv}, we prove that such a numerical scheme preserves nonnegativity of $n$, maximum principle on $\rho$ and monotonicity of $\rho$ ({\it cf.} Proposition~\ref{prop:conv}). 
\begin{prpstn}\label{prop:conv}
%Assuming that $R(y,\rho)\equiv 0$ and 
Consider the scheme~\eqref{scheme:conv} only. If the numerical scheme~\eqref{scheme:conv} is uniquely solvable, then the following properties hold:
\vspace{-0.3cm}
\begin{itemize}
\item[(i)] [nonnegativity] \\ if $\bn^h\ge0$ then $\bn^*\ge0$; \\
\vspace{-0.3cm}
\item[(ii)] [maximum principle on $\rho$] \\
if $0\le \brho^h \le \rho_M$ then $0\le \brho^* \le \rho_M$; \\
\vspace{-0.3cm}
\item[(iii)] [monotonicity of $\rho$] \\
%$\rho_{j+\frac12}^h \le \rho_{j-\frac12}^h$ for all $j\in\llb1, m_x-1\rrb$ and all $h>0$
if $\brho^h$ is monotonically decreasing then $\brho^*$ is monotonically decreasing.
\end{itemize}
\end{prpstn}
\begin{proof}
{\it (i)} The implicit scheme \eqref{scheme:conv} can be rewritten as 
\begin{align}
\label{eqrev1}
-a_{j-1,k}^{*} N_{j-\frac32,k-\frac12}^{*}+
&b_{j,k}^{*} N_{j-\frac12,k-\frac12}^{*} 
-c_{j+1,k}^{*} N_{j+\frac12,k-\frac12}^{*} +\\
&\vep\frac{\Dt}{(\Dy)^2} \big(-N_{j-\frac12,k-\frac32}^{*} + 2N_{j-\frac12,k-\frac12}^{*} - N_{j-\frac12,k+\frac12}^{*} \big)
=
N_{j-\frac12,k-\frac12}^h,
\end{align}
where  
\begin{align*}
&a_{j,k}^{*} = \frac{\Dt}{\Dx}\mu_{k-\frac12}(\delta_x\rho_{j}^{*})_{-} \ge 0,
\quad
c_{j,k}^{*} = \frac{\Dt}{\Dx}\mu_{k-\frac12}(\delta_x\rho_{j-1}^{*})_{+} \ge 0, 
\\
&b_{j,k}^{*} = 1+\frac{\Dt}{\Dx}\mu_{k-\frac12}\left[(\delta_x\rho_{j-1}^{*})_{+}+(\delta_x\rho_{j}^{*})_{-}\right] = 1 + a_{j,k}^{*} + c_{j,k}^{*}.
\end{align*}
The system of equations~\eqref{eqrev1} can be written in matrix form as
\be
\bM^{*} \bn^{*} = \bn^h, \nonumber
\ee
where $\bM^{*}$ is a matrix containing the terms $a_{j,k}^{*}$'s, $b_{j,k}^{*}$'s and $c_{j,k}^{*}$'s with $j \in \llb1,  m_x\rrb$, $k\in \llb1,  m_y\rrb$. 
Since the matrix $\bM^{*}$ is strictly diagonally dominant by columns, it is invertible and all elements of $\left(\bM^{*}\right)^{-1}$ are positive. This ensures that $\bn^{*}$ is nonnegative if $\bn^h$ is nonnegative. \\

{\it (ii)}  Summing \eqref{scheme:conv} over all $k\in \llb 1, m_y \rrb$, we find 
\be\label{scheme:conv_rho}
\frac{\rho_{j-\frac12}^{*} - \rho_{j-\frac12}^h}{\Dt} - \frac{1}{\Dx} \left[ 
<\mu_{k-\frac12} N_{j,k-\frac12}^ {*,\rm{upwind}}>\delta_x \rho_{j}^ {*}
- 
<\mu_{k-\frac12} N_{j-1,k-\frac12}^ {*,\rm{upwind}}>\delta_x \rho_{j-1}^ {*}
\right] 
=
0, 
\ee
where $\delta_x \rho_{j}^ {*} = \left(\rho_{j+\frac12}^ {*}-\rho_{j-\frac12}^ {*}\right) / \Dx$, $<\mu_{k-\frac12} N_{j,k-\frac12}^{*,\rm{upwind}} >= \Dy \sum_{k=1}^{m_y}\mu_{k-\frac12} N_{j,k-\frac12}^{*,\rm{upwind}}$ and 
\be
N_{j,k-\frac12}^{*,\rm{upwind}} = 
\begin{cases}
N_{j-\frac12,k-\frac12}^{*} \quad&\text{ if } \delta_x\rho_j^{*} < 0, \\
N_{j+\frac12,k-\frac12}^{*} \quad&\text{ if } \delta_x\rho_{j}^{*} \ge 0. 
\end{cases}
\ee
For simplicity of notation, we define 
$
\displaystyle{d_j^{*} = \frac{\Dt}{\Dx^2}<\mu_{k-\frac12} N_{j,k-\frac12}^ {*,\rm{upwind}}>}.
$ 
Notice that $d_j^{*}\ge0$. Then, the system of equations~\eqref{scheme:conv_rho} can be rewritten as 
\be\label{scheme:conv_rho_d}
(1 + d_{j-1}^{*} + d_{j}^ {*}) \rho_{j-\frac12}^ {*} -  d_{j-1}^{*}  \rho_{j-\frac32}^ {*}
-d_{j}^{*}  \rho_{j+\frac12}^ {*}
=\rho_{j-\frac12}^ h.
\ee
Assume that $\displaystyle{\rho_{j_0-\frac12}^ {*} = \min_j \{\rho_{j-\frac12}^{*}\}}$, we claim that $\rho_{j_0-\frac12}^{*} \ge0$. 
In fact, if $\rho_{j_0-\frac12}^{*} < 0$, we have
\be
\rho_{j_0-\frac12}^h = \rho_{j_0-\frac12}^{*} + d_{j_0-1}^{*}(\rho_{j_0-\frac12}^{*} - \rho_{j_0-\frac32}^{*}) + d_{j_0}^{*}(\rho_{j_0-\frac12}^{*} - \rho_{j_0+\frac12}^{*}) \le  \rho_{j_0-\frac12}^{*}<0, 
\ee
which is a contradiction. Hence, $\brho^{*}\ge0$. Similarly, one can prove that $\brho^{*} \le \rho_M$. \\

{\it (iii)} Introducing the notation $w_j^{*} = \rho_{j+\frac12}^{*} - \rho_{j-\frac12}^{*}$, we rewrite~\eqref{scheme:conv_rho_d} as 
\be\label{scheme:conv_rho_w}
\rho_{j-\frac12}^{*} + d_{j-1}^{*} w_{j-1}^{*} - d_{j}^{*} w_{j}^{*} = \rho_{j-\frac12}^{h}. 
\ee
Changing all subscripts in \eqref{scheme:conv_rho_w} from $j$ to $j+1$, after a little algebra we find
\be\label{scheme:conv_dw}
(1+2d_j^{*})w_j^{*} - d_{j-1}^{*} w_{j-1}^{*} - d_{j+1}^{*} w_{j+1}^{*} = w_j^h.
\ee
Writing the  system of the equations \eqref{scheme:conv_dw} in matrix form and using arguments similar to those used in part {\it (i)}, it is possible to prove that $\bw^{*}\le0$ if $\bw^{h}\le0$.
\end{proof}

\paragraph{Properties of the numerical scheme~\eqref{scheme:rho_growth}} The numerical scheme~\eqref{scheme:rho_growth} satisfies the properties established by Proposition~\ref{prop:grow}.
\begin{prpstn}\label{prop:grow}
Consider the scheme~\eqref{scheme:rho_growth} only. If $R(y,\rho)$ satisfies assumptions~\eqref{ass:R} and $\bn^*\ge0$, then the following properties hold:
\vspace{-0.3cm}
\begin{itemize}
\item[(i)] [existence and uniqueness and nonnegativity] \\ the scheme~\eqref{scheme:rho_growth} admits a unique solution such that $\bn^{h+1}\ge0$; \\
\vspace{-0.3cm}
\item[(ii)]  [maximum principle on $\rho$]\\ if $0\le \brho^* \le \rho_M$, then $0 \le \brho^{h+1} \le \rho_M$.
\end{itemize}
\end{prpstn}

\begin{proof}
{\it (i)} It is sufficient to prove existence and uniqueness of $\rho_{j-\frac12}^{h+1}$. Let
$$
f(\rho) = \rho -  \Dy \sum_{k=1}^{m_y} \exp\left(\frac{U_{j-\frac12,k-\frac12}^{*}+\Dt R(y_{k-\frac12}, \rho)}{\vep}\right).
$$ 
Since $f'(\rho) >0$, $f(0)<0$ and $\displaystyle{\lim_{\rho\to\infty} f(\rho) =\infty}$, equation \eqref{eq:rho_growth} has a unique positive root, which is $\rho_{j-\frac12}^{h+1}$. From this,
existence, uniqueness and nonnegativity of $N_{j-\frac12,k-\frac12}^{h+1}$ immediately follow. \\
%since it can be computed explicitly once that $\rho_{j-\frac12}^{h+1}$ is known.  Proving positivity of $N_{j-\frac12,k-\frac12}^{h+1}$ is trivial.\\

{\it (ii)} Noticing that  $f'(\rho)>0$, $f(0) <0$ and 
\be
f(\rho_M) \ge \rho_M -  \Dy \sum_{k=1}^{m_y} \exp\left(\frac{U_{j-\frac12,k-\frac12}^{*}}{\vep}\right) = \rho_M -\rho_{j-\frac12}^* \ge0,
\ee
we conclude that equation \eqref{eq:rho_growth} has a unique solution in the interval $[0, \rho_M]$. This implies that $0\le \brho^{h+1} \le \rho_M$. 
\end{proof}

\section*{Acknowledgements}
B.P. has received funding from the European Research Council (ERC) under the European Union's Horizon 2020 research and innovation programme (grant agreement No 740623). T.L. gratefully acknowledges support of the project PICS-CNRS no. 07688 and the MIUR grant ``Dipartimenti di Eccellenza 2018-2022'', and would like to thank Alexander Lorz for insightful discussions during the early stages of the project.

\bibliographystyle{siam}
\bibliography{BibliographyBTX}

\end{document}